\documentclass[12pt]{article}
\usepackage{amsmath,amssymb,amsthm,graphicx}
\usepackage{enumitem}
\usepackage{hyperref}
\usepackage{fullpage}
\usepackage{float}
\usepackage{setspace}
\usepackage{tabu}

\newtheorem{theorem}{Theorem}
\newtheorem{lemma}[theorem]{Lemma}

\newtheorem{corollary}[theorem]{Corollary}
\newtheorem{conjecture}[theorem]{Conjecture}

\begin{document}

\title{Meyniel's conjecture on graphs of bounded degree}

\author{
Seyyed Aliasghar Hosseini
\and Bojan Mohar
\and Sebastian Gonzalez Hermosillo de la Maza
}

\date{\today}

\maketitle

\begin{abstract}
The game of Cops and Robbers is a well known pursuit-evasion game played on graphs. It has been proved \cite{bounded_degree} that cubic graphs can have arbitrarily large cop number $c(G)$, but the known constructions show only that the set $\{c(G) \mid G \text{ cubic}\}$ is unbounded. In this paper we prove that there are arbitrarily large subcubic graphs $G$ whose cop number is at least $n^{1/2-o(1)}$ where $n=|V(G)|$. We also show that proving Meyniel's conjecture for graphs of bounded degree implies a weaker version of Meyniel's conjecture for all graphs.
\end{abstract}

\section{Introduction}
Cops and Robbers is a Pursuit-evasion game played on graphs with two players, one controls the cops and the other one controls the robber. The game begins by the cops selecting some vertices as their initial positions. Then the robber, knowing the positions of cops, selects his initial vertex. From now on, first the cops move and then the robber moves, where moving means going to a neighboring vertex or staying at the same position.
The goal of the cops is to capture the robber, which means having a cop at the same vertex as the robber, and the goal of the robber is to prevent this from happening.
The minimum number of cops that guarantee the robber's capture in a graph $G$ is called the \emph{cop number} of $G$ and is denoted by $c(G)$.

One of the most important open problems in this area is \emph{Meyniel's Conjecture}  (communicated
by Frankl \cite{frankl}).

\begin{conjecture}
For a connected graph $G$, if $n=|V(G)|$, then $c(G)=O(\sqrt{n}\,)$.
\end{conjecture}

This conjecture has received a lot of attention but it is still far away from being proved. In fact the following conjecture, called \emph{Weak Meyniel's Conjecture}\footnote{Also known as Soft Meyniel's Conjecture.}, is still widely open.

\begin{conjecture}\label{conj:weak Meyniel}
There is an $\varepsilon>0$ such that every graph $G$ of order $n$ has $c(G)=O(n^{1-\varepsilon})$.
\end{conjecture}

Suppose that Conjecture \ref{conj:weak Meyniel} holds for every $\varepsilon$ such that $1-\varepsilon > \delta$. Then we say that the \emph{Weak Meyniel's Conjecture holds with exponent $\delta$}.

\section{Degree reduction}

In this section, we show that from every graph $G$ we can construct a graph $\widehat G$ with the following properties:
\begin{itemize}
\item[\rm (a)]
$c(\widehat G) \ge c(G)$,
\item[\rm (b)]
$\widehat G$ has smaller maximum degree than $G$, and
\item[\rm (c)]
$\widehat G$ is not much larger than $G$.
\end{itemize}
In the construction, we will replace each vertex of $G$, whose degree is $d$, with a copy of the following graph $A_d = A_d(m)$, where $2\le m\le d$. To form the graph $A_d$, we start with a set $X = \{x_1, \ldots, x_d\}$ of $d$ mutually nonadjacent vertices that are partitioned into $m$ almost equal parts $X_1,\dots, X_m$ (so that $\lfloor d/m\rfloor \le |X_i| \le \lceil d/m\rceil$ for $1\le i \le m$). Take $\binom{m}{2}$ additional vertices $y_{ij}=y_{ji}$ for $1\le i < j \le m$ and join each $y_{ij}$ to all vertices in $X_i\cup X_j$. The resulting bipartite graph $A_d$ has $d+\binom{m}{2}$ vertices, each $y_{ij}$ has degree between $2d/m$, and each $x_t$ has degree $m-1$. See Figure \ref{fig:A10_4} for an example.

\begin{figure}[H]
    \centering
    \includegraphics[width=0.5\textwidth]{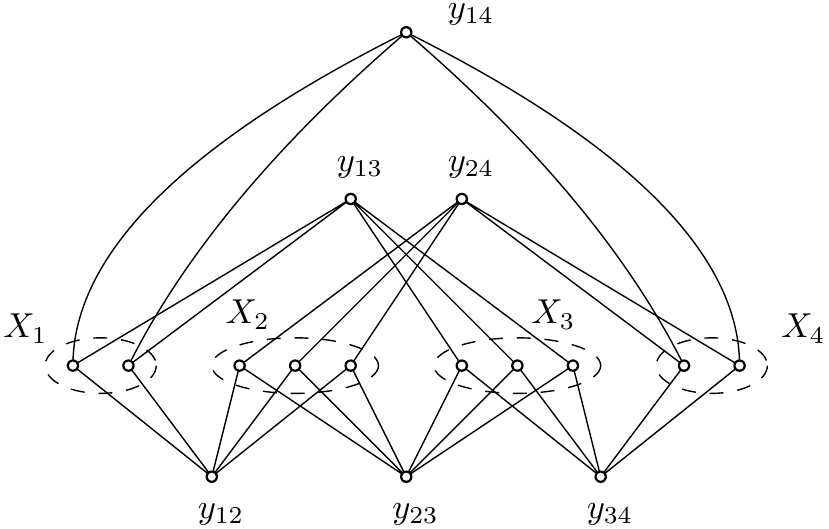}
    \caption{The replacement graph $A_{10}(4)$ for degree 10 with 4 parts.}
    \label{fig:A10_4}
\end{figure}

In the following theorem we replace each vertex of a graph with a subgraph of the form $A_d(m)$, where $d$ is the degree of the vertex, and give a construction of a graph $\widehat G$ with desired properties (a)--(c). Recall that $\Delta(G)$ denotes the maximum vertex degree in $G$.

\begin{theorem}\label{thm:decrease degree}
For every graph $G$, there is a graph $\widehat G$ with the following properties:
\begin{itemize}
\item[\rm (a)]
$c(\widehat G) \ge c(G)$,
\item[\rm (b)]
$\Delta(\widehat G) \le \left\{
                          \begin{array}{ll}
                            3, & \hbox{if $\Delta(G)\le 4$;} \\[1mm]
                            2\Bigl\lceil \sqrt{\Delta(G)/2}\,\Bigr\rceil, & \hbox{otherwise.}
                          \end{array}
                        \right.$
\item[\rm (c)]
$|V(\widehat G)| \le \tfrac{11}{5}\Delta(G)|V(G)|$.
\end{itemize}
\end{theorem}

\begin{figure}[H]
    \centering
    \includegraphics[width=0.75\textwidth]{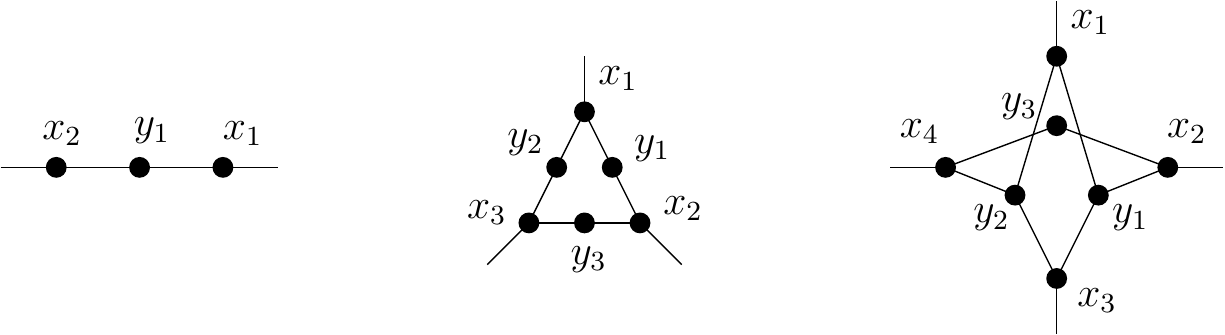}
    \caption{The replacement graphs $A_d$ for $d=2,3,4$ with added semi-edges incident with vertices $x_1,\dots,x_d$.}
    \label{fig:A432}
\end{figure}

\begin{proof}
Let $v$ be a vertex of degree $d$ in $G$ and let $w_1, \ldots, w_d$ be its neighbors. Replace $v$ by a copy of the gadget $A_d$ described above and join $x_i$ to $w_i$ for $i=1,\ldots, d$. Apply this change to all vertices of $G$ to get $\widehat{G}$. After doing this, each edge $uv$ of $G$ has been replaced by an edge joining two $x$-vertices in the corresponding gadgets $A_d$ and $A_{d'}$ corresponding to $u$ and $v$. 

\medskip
(a) The idea that we will use in this part is similar to showing that subdividing all edges of a graph the same number of times does not decrease the cop number.
Let $t = c(G)-1$. Then $t$ cops cannot capture the robber in $G$. Consider the corresponding escape strategy. We will use it to show that the robber can escape from $t$ cops in $\widehat{G}$, thus confirming that $c(\widehat{G})\ge c(G)$.

At the beginning of the game each cop will choose a vertex (in $\widehat{G}$) as their initial position. Each vertex is in a gadget $A_d$, which corresponds to a vertex $u$ in $G$. The vertex $u$ will be called the \emph{shadow} of the vertices in this copy of $A_d$.  The robber will assume that each cop is in the corresponding shadow vertex in $G$ and will play the escape strategy in $G$. Initially, he has a vertex $v$ to pick in $G$, and he will pick the vertex $x_1\in V(A_k)$ in $\widehat{G}$. From now on, the robber will move based on the movements of (the shadows of) cops in $G$. Whenever the shadows of cops have moved in $G$, the robber has an escaping move in $G$ which can be translated into a series of at most 3 moves in $\widehat{G}$. Since the shadow of the robber will not be captured in the next move of cops in $G$, the robber will not get captured in $\widehat{G}$ in the next 3 moves. During these three moves of the robber, the cops will also make 3 moves, but the shadow of each cop in $G$ will either stay the same or move to an adjacent vertex in $G$. Thus, the robber can interpret the change of the shadows as the moves of the cops in $G$. It will be his turn to move, so he can continue using the escaping strategy for $G$. Therefore the robber in $\widehat{G}$ can copy the strategy of the robber in $G$ and escape from $t$ cops in $\widehat{G}$.

\medskip
(b)
The statement is clear when $\Delta \leq 4$, see Figure \ref{fig:A432}. In order to prove the stated bound for $\Delta \geq 5$, we take $m = m(d) = \bigl\lceil \sqrt{2d}\,\bigr\rceil$. For this choice of $m$, the degrees of vertices $x_i$ and $y_{ij}$ in $A_d(m)$ are all approximately equal to each other. The degree of each $x_i$ is equal to $m$ and the degree of any $y_{ij}$ is at most
$$
   2\biggl\lceil \frac{d}{\bigl\lceil\sqrt{2d}\,\bigr\rceil} \,\biggr\rceil \leq
   2\biggl\lceil \frac{d}{\sqrt{2d}} \,\biggr\rceil =
   2\Bigl\lceil \sqrt{d/2} \,\Bigr\rceil.
$$

\medskip
(c)
Let $\Delta=\Delta(G)$. First note that if $\Delta\leq 3$, then $G$ will have the desired properties itself. So we may assume that $\Delta\geq 4$. If $\Delta=4$, then we are replacing each vertex with at most 7 vertices, thus $|\widehat{G}| \le 7|G| \le \tfrac{11}{5}\Delta|G|$, as desired.

For vertices of degree $d\geq 5$, we consider $m$ to be $\bigl\lceil \sqrt{2d}\,\bigr\rceil$ as in part (b). Then $A_d(m)$ has $d+\binom{m}{2}$ vertices. Using $\bigl\lceil \sqrt{2d}\,\bigr\rceil\leq \sqrt{2d} +1$, it is easy to see that $d+\binom{m}{2}\le \tfrac{11}{5} d$ for $d\ge 13$, and a direct computation shows that the same bound holds for $d=5,\dots,13$. Thus, $|\widehat G| \le \tfrac{11}{5}\Delta|G|$. Note that for $d=5$ we get the equality.
\end{proof}

Andreae \cite{bounded_degree} was the first one to prove that cubic graphs can have arbitrarily large cop number. Explicit examples are cubic graphs of large girth which were shown to have large cop number by Frankl \cite{frankl}. 
By repeatedly using Theorem \ref{thm:decrease degree}, we obtain a new proof of this fact.


These results are also applicable to digraphs, see Section \ref{sec:5.3}.

\section{Reducing the degree further}

Using Theorem \ref{thm:decrease degree} repeatedly, we can decrease the degree of each vertex down to 3.
To analyse the number of steps needed, let us assume that $G$ is a graph on $n$ vertices with $c(G)=c$ and $\Delta(G)=d$. We also let $\mu=\tfrac{11}{5}$. Then we have:

\begin{center}
\begin{tabu} to 0.8\textwidth { | X[c] | X[c] | X[c] | X[c]| }
 \hline
 Graph & $\#$ vertices & $\Delta$ & cop number \\
 \hline
 $G$ & $n$ & $d$ & $c$\\
 $\widehat{G}$ & $\leq \mu d n$ & $\bigl\lceil \sqrt{2d}\,\bigr\rceil$ & $\geq c$ \\
 \hline
\end{tabu}
\end{center}
By repeating the strategy to get $\widehat{G}'$ from $\widehat{G}$ we will have:
\begin{center}
\begin{tabu} to 0.8\textwidth { | X[c] | X[c] | X[c] | X[c]| }
 \hline
 $\widehat{G}'$ & $\leq \mu^2 d\bigl\lceil \sqrt{2d}\,\bigr\rceil n$ & $\biggl\lceil \sqrt{2\bigl\lceil \sqrt{2d}\,\bigr\rceil}\,\biggr\rceil$ & $\geq c$ \\
 \hline
\end{tabu}
\end{center}

If $2^{2^{k-2}+1}< d \leq 2^{2^{k-1}+1}$, then for $\widehat{d}$, the maximum degree in $\widehat{G}$, we have $2^{2^{k-3}+1}< \widehat{d} \leq 2^{2^{k-2}+1}$. Therefore, in $k$ steps the maximum degree of the graph will become 3. 
In other words, by repeating the argument $k\leq \log \log \frac{d}{2} +2$ times \footnote{All logarithms in this paper are taken base 2.}, the degrees of all vertices will be at most 3 and the number of vertices of the graph will be at most
$$\mu^k \cdot 2^{1/2+1/4+\cdots +1/2^{k-1}} \cdot d^{1+1/2+1/4+\cdots+1/2^k} n\leq 2\mu^k\cdot d^2 \cdot n\leq O(d^2n\log^{1.14}{d}).$$
This yields the following corollary.

\begin{corollary}\label{cor:degree}
Let\/ $G$ be a graph on $n$ vertices with maximum degree $d$, then there exist a subcubic graph $H$ on $O(d^2n\log^{1.14}{d})$ vertices such that $c(H)\geq c(G)$.
\end{corollary}

Let $\omega:\mathbb{N}\rightarrow \mathbb{N}$ be a function that tends to infinity, say $\omega(n)=5\log{n}$,
and let
$G\in \mathcal{G}_{n,p}$ for $p=\frac{\omega(n)}{n}$. It has been proved in \cite{random, zigzag} that as long as $pn > 2.1 \log n$, we have asymptotically almost surely (a.a.s.)
$$
c(G)\geq \frac{1}{(pn)^2}n^{\frac{1}{2}-\frac{9}{2\log\log{pn}}} = \frac{1}{(\omega(n))^2}n^{\frac{1}{2} - \frac{9}{2\log \log \omega(n)}} = \Omega(n^{1/2-\varepsilon}).
$$

In this graph the maximum degree is a.a.s. at most $2pn=2\omega(n)$. Applying the above approach we will get $G'$, a graph on $O(n\,\omega^2(n)\log^{1.14}{\omega(n)})$ vertices with $c(G')\geq \Omega(n^{1/2-\varepsilon})$. Therefore, by a change of variable it is easy to check that if $G'$ is a (subcubic) graph on $n$ vertices then $c(G')\geq \Omega(n^{1/2-\varepsilon})$.

\begin{theorem}\label{thm:degree}
For every value of $\varepsilon>0$ and large enough $n$, there are subcubic graphs on $n$ vertices with cop number $\Omega(n^{1/2-\varepsilon})$.
\end{theorem}

\section{Digraphs of bounded degree}
\label{sec:5.3}

Recently, the Meyniel's Conjecture on digraphs has received a lot of attention \cite{high-girth,Jeremie,hosseini_game_2018,grid}. In this section we will use the same technique as above to find Eulerian digraphs of bounded degree and with arbitrarily large cop number. The gadget that we are going to get is different but will have the same properties.

\begin{theorem}\label{thm:digraph}
For every graph $D$, there is a graph $D'$ with the following properties:
\begin{itemize}
\item[\rm (a)]
$c(D') \ge c(D)$,
\item[\rm (b)]
$\Delta^+(D')= \Delta^-(D') = 2$
\item[\rm (c)]
$|V(D')| \le 4(\Delta^- + \Delta^+)|V(D)|$.
\end{itemize}
\end{theorem}

\begin{proof} Let the maximum out-degree of the digraph be $\Delta^+=o$ and the maximum in-degree $\Delta^-=i$.  Consider $x^-_1, \ldots, x^-_i$ and  $x^+_1, \ldots, x^+_o$. Now find the smallest $k$ such that $2^k\geq i$ and make a complete binary tree where $x^-_j$ are the leaves of the tree and direct all edges towards the root.
Also find the smallest $l$ such that $2^l\geq o$ and make a complete binary tree whose leaves are $x^+_j$ and direct all edges away from the root. Now merge the roots of these directed trees.
Note that this gadget has less than $2^{k+1}+2^{l+1}\leq 4(o+i)$ vertices and the distance from any $x^-_s$ to any $x^+_e$ is fixed and equal to $k+l$ (for $1\leq s\leq i$ and $1\leq e\leq o$).

\begin{figure}[h]
    \centering
    \includegraphics[width=0.45\textwidth]{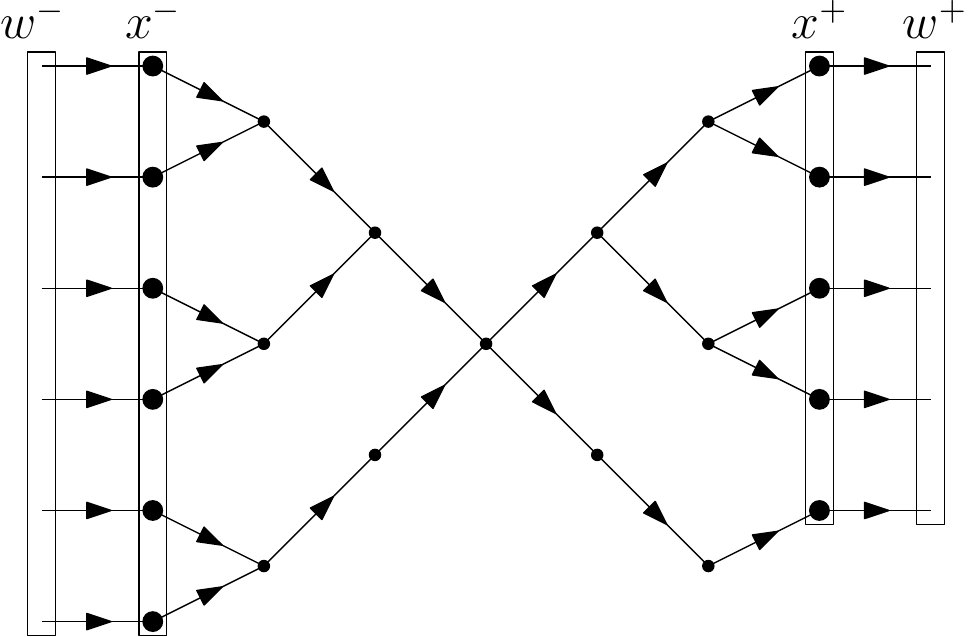}
    \caption{The gadget when $\Delta^-=6$ and $\Delta^+=5$ (after deleting extra vertices).}
    \label{digraph-bounded-degree}
\end{figure}

Now for any vertex $v$ where $w^-_1, \ldots, w^-_{d^-(v)}$ and $w^+_1, \ldots, w^+_{d^+(v)}$ are in and out-neighbors of $v$, connect (with a directed edge) $w^-_j$ to $x^-_j$ ($j=1,\ldots, d^-(v)$) and similarly connect $x^+_j$ to $w^+_j$ ($j=1,\ldots,d^+(v)$). See Figure \ref{digraph-bounded-degree}. Note that we can delete the unnecessary vertices.

If we replace all vertices of a digraph $D$ with this gadget to get $D'$, it is easy to see that in and out-degree of vertices of $D'$ is bounded by 2, number of vertices of $D'$ is at most $4(\Delta^- + \Delta^+)$ times the number of vertices of $D$ and by a similar lemma as Lemma \ref{cor:degree} we have $c(D')\geq c(D)$.
\end{proof}

\section{Connection to Meyniel's Conjecture}

We say that a family of graphs \emph{satisfies the Weak Meyniel's Conjecture with exponent $\delta$} if $c(G)\le n^{\delta+o(1)}$ for every graph $G$ in the family, where $n=|V(G)|$ and the asymptotics are with respect to $n$.

\begin{lemma}
Meyniel's Conjecture for subcubic graphs implies the Weak Meyniel's Conjecture with exponent $\tfrac{3}{4}$ for the general case.
\end{lemma}

\begin{proof}
Assume that Meyniel's conjecture is true for subcubic graphs and let $G$ be a general graph on $n$ vertices. 
For a fixed $0<\varepsilon<1$, if $G$ has a vertex of degree at least $n^\varepsilon$, then we will put a cop on it to cover the vertex and its neighbourhood. Remove this vertex and its neighbourhood from $G$ and repeat it until there is not vertex of degree at least $n^\varepsilon$ in $G$. So by using at most $O(n^{1-\varepsilon})$ cops we can reduce the game to a graph of maximum degree at most $n^\varepsilon$. Note that one of the components of this new graph contains all the vertices that the robber visit without being captured. From this graph we can get $G'$, a subcubic graph on $n'=O(n^{1+2\varepsilon}\log^{1.14}{n})$ vertices. By Meyniel's conjecture (for subcubic graphs)
$c(G')\leq O(\sqrt{n'})$.
Therefore we have
$c(G)\leq O(n^{1-\varepsilon}) + O(n^{1/2+\varepsilon} \log^{0.57}{n}).$

Considering $\varepsilon=1/4$, we get
$c(G)\leq O(n^{3/4+o(1)})$.
\end{proof}

Pra\l{}at and Wormald \cite{random2, random3} proved that Meyniel's Conjecture holds for almost all cubic graphs which strongly supports that the weak conjecture holds with exponent 3/4.

The same proof yields the following more general relationship.

\begin{corollary}
Weak Meyniel's conjecture for subcubic graphs with exponent $1-\varepsilon$ implies the Weak Meyniel's conjecture with exponent $1 - \tfrac{1}{2}\varepsilon$ for the general case.
\end{corollary}

\bibliographystyle{plain}
\bibliography{references}

\end{document}